\documentclass[11pt,reqno]{amsart}
\usepackage{enumerate, latexsym, amsmath, amsfonts, amssymb, amsthm, color}

\def\Z{\mathbb Z}

\def\Q{\mathbb Q}

\def\l{\left}
\def\r{\right}
\def\bg{\bigg}
\def\({\bg(}
\def\){\bg)}
\def\t{\text}
\def\f{\frac}

\def\sign{{\rm sign}}

\def\ls{\leqslant}
\def\gs{\geqslant}

\def\la{\lambda}

\def\adj{{\rm adj}}

\def\diag{{\rm diag}}
\def\rank{{\rm rank}}
\def\sign{{\rm sign}}

\def\u{{\mathbf u}}
\def\v{{\mathbf v}}

\def\Proof{\noindent{\it Proof}}
\def\Ack{\medskip\noindent {\bf Acknowledgment}}

\theoremstyle{plain}
\newtheorem{theorem}{Theorem}

\newtheorem{lemma}{Lemma}
\newtheorem{proposition}{Proposition}
\newtheorem{corollary}{Corollary}

\theoremstyle{definition}

\theoremstyle{remark}
\newtheorem{remark}{Remark}

\makeatletter
\@namedef{subjclassname@2020}{%
  \textup{2020} Mathematics Subject Classification}
\makeatother

\allowdisplaybreaks

 \vspace{4mm}

\begin{document}
\hbox{Preprint, {\tt arXiv:2605.16240}}
\medskip

\title
[Evaluation of two determinants involving $q$-integers]
{Evaluation of two determinants \\ involving $q$-integers}

\author
[Zhi-Wei Sun] {Zhi-Wei Sun}

\address{School of Mathematics, Nanjing
University, Nanjing 210093, People's Republic of China}
\email{zwsun@nju.edu.cn}

\keywords{Determinants, $q$-integers, the floor function, the ceiling function, Jacobi symbols.
\newline \indent Supported by the Natural Science Foundation of China (grant 12371004).}

\begin{abstract} The $q$-analogue of an integer $m$ is given by $[m]_q=(1-q^m)/(1-q)$.
Let $a$ be an integer, and let $n$ be a positive odd integer. Via discrete Fourier transforms, we
establish the following two identities:
$$\det\left[\left[\left\lfloor\frac{aj-(a+1)k}n\right\rfloor\right]_q\right]_{1\leqslant j,k\leqslant n}=-\left(\frac{a(a+1)}n\right)q^{(1-3n)/2}$$
and
$$\det\left[\left[\left\lceil\frac{(a+1)j-ak}n\right\rceil\right]_q\right]_{1\leqslant j,k\leqslant n}=\left(\frac{a(a+1)}n\right)q^{(n-1)/2},$$
where $(\frac{\cdot}n)$ denotes the Jacobi symbol.
\end{abstract}
\maketitle

\section{Introduction}
\setcounter{lemma}{0}
\setcounter{theorem}{0}
\setcounter{corollary}{0}
\setcounter{remark}{0}
\setcounter{equation}{0}

For any real number $x$, let $\lfloor x\rfloor $
denote the largest integer not exceeding $x$.
The function $\lfloor \cdot\rfloor$ is called the {\it floor function}.
In 2021, the author \cite{S21} evaluated some permanents and determinants with entries involving the floor function. In 2025, S. Fu, Z. Lin and the author \cite{FLS} proved 
that for any positive integer
$n$ the permanent of the matrix $[\lfloor\f{2j-k}n\rfloor]_{1\ls j,k\ls n}$
is $2(2^{n+1}-1)B_{n+1}$, where $B_0,B_1,B_2,\ldots$ are the Bernoulli numbers.

For any integer $m$, its $q$-analogue is given by
$$[m]_q:=\f{1-q^m}{1-q}=-q^m\f{1-q^{-m}}{1-q}.$$
In particular, $[0]_q=0$, $[1]_q=1$, and $[2]_q=1+q$.
Note that $\lim_{q\to1}[m]_q=m$ for all $m\in\Z$.

For a matrix $A=[a_{jk}]_{1\ls j,k\ls n}$ over a field,
we denote the determinant of $A$ by $\det(A)$ or $\det[a_{jk}]_{1\ls j,k\ls n}$.
In this paper, we mainly study certain determinants involving $q$-integers.

Now we state our first theorem.

\begin{theorem} \label{Th1.1} Let $a$ be an integer and let $n$ be a positive odd integer.
Then we have
\begin{equation}\label{floor}\det\l[\l[\l\lfloor\f{aj-(a+1)k}n\r\rfloor\r]_q\r]_{1\ls j,k\ls n}=-\l(\f{a(a+1)}n\r)q^{(1-3n)/2},
\end{equation}
where $(\f{\cdot}n)$ is the Jacobi symbol.
\end{theorem}
\begin{remark} Jacobi symbols play important roles in the theory of quadratic residues modulo primes.
For their definition and basic properties, one may consult \cite[pp.\, 56-57]{IR}.
\end{remark}

Letting $q$ tend to $1$ or $2$, we obtain from Theorem \ref{Th1.1} the following corollary.

\begin{corollary} For any integer $a$ and positive odd integer $n$, we have
\begin{equation}\label{1floor}\det\l[\l\lfloor\f{aj-(a+1)k}n\r\rfloor\r]_{1\ls j,k\ls n}=-\l(\f{a(a+1)}n\r)
\end{equation}
and
\begin{equation}\label{2floor}\det\l[2^{\lfloor\f{aj-(a+1)k}n\rfloor}-1\r]_{1\ls j,k\ls n}=-\l(\f{a(a+1)}n\r)2^{(1-3n)/2}.
\end{equation}
\end{corollary}

The ceiling function $\lceil\cdot\rceil$ is defined as follows: For any real number $x$, $\lceil x\rceil$
denotes the least integer not smaller than $x$. Our second theorem involve $q$-integers and the ceiling function.

\begin{theorem} \label{Th1.2} Let $a$ be an integer and let $n$ be a positive odd integer.
Then we have
\begin{equation}\label{ceil}\det\l[\l[\l\lceil\f{(a+1)j-ak}n\r\rceil\r]_q\r]_{1\ls j,k\ls n}=\l(\f{a(a+1)}n\r)q^{(n-1)/2}.
\end{equation}
\end{theorem}

Letting $q$ tend to $1$ or $2$, we obtain from Theorem \ref{Th1.1} the following corollary.

\begin{corollary} For any integer $a$ and positive odd integer $n$, we have
\begin{equation}\label{1ceil}\det\l[\l\lceil\f{(a+1)j-ak}n\r\rceil\r]_{1\ls j,k\ls n}=\l(\f{a(a+1)}n\r)
\end{equation}
and
\begin{equation}\label{2ceil}\det\l[2^{\lceil\f{(a+1)j-ak}n\rceil}-1\r]_{1\ls j,k\ls n}=\l(\f{a(a+1)}n\r)2^{(n-1)/2}.
\end{equation}
\end{corollary}

Both Theorems \ref{Th1.1} and \ref{Th1.2} were conjectured by the author in 2021 (cf. \cite{MO} and \cite{S21}), and they remain open for nearly five years.
We will deduce an auxiliary proposition in Section 2
via the discrete Fourier transforms, and then prove Theorems \ref{Th1.1} and \ref{Th1.2} in Section 3. 

Our following theorem is more general than Theorems \ref{Th1.1} and \ref{Th1.2}. We will also include its proof in Section 3.

\begin{theorem}\label{Th1.3} Let $a$ be any integer, and let $n$ be a positive odd integer.
Then
\begin{equation}\label{q-floor}\det\l[x+q^{\lfloor \f{aj-(a+1)k}n\rfloor}\r]_{1\ls j,k\ls n}
=\l(\f{a(a+1)}n\r)q^{(1-3n)/2}(1-q)^{n-1}(1+qx)
\end{equation}
and
\begin{equation}\label{q-ceil}\det\l[x+q^{\lceil \f{(a+1)j-ak}n\rceil}\r]_{1\ls j,k\ls n}
=\l(\f{a(a+1)}n\r)q^{(n-1)/2}(1-q)^{n-1}(x+q).
\end{equation}
\end{theorem}

Taking $x=0$ and $q\in\{-1,2\}$ in Theorem \ref{Th1.3}, we obtain the following corollary.

\begin{corollary} Let $a$ be an integer, and let $n$ be a positive odd number. Then
\begin{equation}\begin{aligned}\label{-1-floor}\det\l[(-1)^{\lfloor \f{aj-(a+1)k}n\rfloor}\r]_{1\ls j,k\ls n}
=&\ \det\l[(-1)^{\lceil \f{(a+1)j-ak}n\rceil}\r]_{1\ls j,k\ls n}
\\=&\ \l(\f{a(a+1)}n\r)(-1)^{(n+1)/2}2^{n-1}.
\end{aligned}\end{equation}
Also,
\begin{equation}\label{2-floor}\det\l[2^{\lfloor \f{aj-(a+1)k}n\rfloor}\r]_{1\ls j,k\ls n}
=\l(\f{a(a+1)}n\r)2^{(1-3n)/2}
\end{equation}
and
\begin{equation}\label{2-ceil}\det\l[2^{\lceil \f{(a+1)j-ak}n\rceil}\r]_{1\ls j,k\ls n}
=\l(\f{a(a+1)}n\r)2^{(n+1)/2}.
\end{equation}
\end{corollary}

Throughout this paper, for any real number $x$ we call $\{x\}=x-\lfloor x\rfloor$ the {\it fractional part} of $x$. For a matrix $A=[a_{ij}]_{1\ls i,j\ls n}$ over a field,
we use $A^T$ to denote the transpose of $A$, and $\adj(A)$ to denote the {\it adjugate matrix} of $A$ giving by $\adj(A)=[A_{ji}]_{1\ls i,j\ls n}$, where $A_{ji}$ is the cofactor of the entry $a_{ji}$ in $A$. For complex numbers $c_1,\ldots,c_n$, by $\diag(c_1,\ldots,c_n)$ we mean the diagonal matrix
$[c_{jk}]_{1\ls j,k\ls n}$ with $c_{jj}=c_j$ for all $j=1,\ldots,n$.

\section{An Auxiliary Proposition}
\setcounter{lemma}{0}
\setcounter{theorem}{0}
\setcounter{corollary}{0}
\setcounter{remark}{0}
\setcounter{equation}{0}

We need a well known lemma on discrete Fourier transforms which can be easily proved.

\begin{lemma}\label{Lem2.1} Let $a_1,\ldots,a_n,b_1,\ldots,b_n$ be complex numbers. Then
$$a_m=\sum_{k=1}^n b_ke^{2\pi i km/n}\ \t{for all}\ m=1,\ldots,n$$
if and only if
$$b_m=\sum_{k=1}^n a_ke^{2\pi i km/n}\ \t{for all}\ k=1,\ldots,n.$$
\end{lemma}

The following lemma is Frobenius' extension (cf. \cite{BC15}) of the Zolotarev lemma \cite{Z}.

\begin{lemma}\label{Lem2.2} Let $n$ be a positive odd integer, and let $a\in\Z$ be relatively prime to $n$.
For $j=1,\ldots,n$, let $\lambda_a(j)$
be the least positive residue of $aj$ modulo $n$. Then the permutation $\lambda_a$ of $\{1,\ldots,n\}$ has the sign $\sign(\lambda_a)=(\f an)$.
\end{lemma}

\begin{remark} C. Huang and H. Pan \cite{HP} established a new result similar to Lemma \ref{Lem2.2}.
\end{remark}

\begin{lemma} \label{Lem2.3} Let $\zeta$ be any primitive $n$th root of unity. Then
$$[\zeta^{jk}]_{1\ls j,k\ls n}^{-1}=\f1n[\zeta^{-jk}]_{1\ls j,k\ls n}.$$
\end{lemma}
\Proof. Let $A=[\zeta^{jk}]_{1\ls j,k\ls n}$ and $B=[\zeta^{-jk}]_{1\ls j,k\ls n}$.
For $j,k\in\{1,\ldots,n\}$, the $(j,k)$-entry of $C=AB$ is
$$c_{jk}=\sum_{m=1}^n\zeta^{jm}\zeta^{-mk}=\sum_{m=0}^{n-1}(\zeta^{j-k})^m
=\begin{cases}n&\t{if}\ j=k,
\\\f{1-(\zeta^{j-k})^n}{1-\zeta^{j-k}}=0&\t{if}\ j\not=k.
\end{cases}$$
So we have $A^{-1}=\f1n B$. This ends the proof. \qed

For convenience, for a statement $S$ we set
$$[\![S]\!]=\begin{cases}1&\t{if}\ S\ \t{holds},
\\0&\t{otherwise}.\end{cases}$$

\begin{proposition}\label{Prop2.1} Let $n$ be a positive odd integer and let $a$ be any integer.
Let $q\not=1$ be a positive real number, and define the matrix
\begin{equation}\label{Q}Q=\l[q^{-\{\f{aj-(a+1)k}n\}}\r]_{1\ls j,k\ls n}.
\end{equation}
Then
\begin{equation}\label{detQ}\det(Q)=\l(\f{a(a+1)}n\r)(1-q^{-1})^{n-1}.
\end{equation}
Moreover, when $\gcd(a(a+1),n)=1$ we have
$$Q^{-1}=\f 1{1-q^{-1}}[f(j,k)]_{1\ls j,k\ls n},$$
where
$$f(j,k)=[\![n\mid(a+1)j-ak]\!]-q^{-1/n}[\![n\mid (a+1)j-ak-1]\!].$$

\end{proposition}
\Proof. Let $\zeta=e^{2\pi i/n}$. For any $x\in\Z$, by Lemma \ref{Lem2.1} we have
$$q^{-\{x/n\}}=\sum_{m=1}^n c_m\zeta^{mx},$$
where
$$c_m=\f1n\sum_{x=1}^{n}q^{-\{x/n\}}\zeta^{-mx}=\f1n\sum_{k=0}^{n-1} q^{-k/n}\zeta^{-km}=\f{1-q^{-1}}{n(1-q^{-1/n}\zeta^{-m})}$$
for all $m=1,\ldots,n.$ Thus, for any $j,k=1,\ldots,n$,
$$q^{\{-(aj-(a+1)k)/n\}}=\sum_{m=1}^n c_m\zeta^{(aj-(a+1)k)m}=\sum_{m=1}^n \zeta^{ajm}c_m\zeta^{-(a+1)mk}.$$
So $Q=UCV$, where $U=[\zeta^{ajk}]_{1\ls j,k\ls n}$, $V=[\zeta^{-(a+1)jk}]_{1\ls j,k\ls n}$,
and $C=\diag(c_1,\ldots,c_n)$. Hence
\begin{equation}\label{UCV}\det(Q)=\det(U)\det(C)\det(V)=c_1\cdots c_n\det(U)\det(V)
\end{equation}
If $\gcd(a,n)>1$ then $j=n/\gcd(a,n)\in\{1,\ldots,n-1\}$ and $\zeta^{aj}=a\zeta^{an}$.
So $\det(U)=0$ if $\gcd(a,n)>1$. Similarly, $\det(V)=0$ if $\gcd(a+1,n)>1$.
Therefore both sides of \eqref{detQ} vanish when $\gcd(a(a+1),n)>1$.

Below we assume that $\gcd(a(a+1),n)=1$. For $j=1,\ldots,n$, let $\la_a(j)$
be the least positive residue of $aj$ modulo $n$. Then
\begin{align*}\det(U)&=\sum_{\sigma\in S_n}\sign(\sigma)\prod_{j=1}^n\zeta^{aj\sigma(j)}
=\sum_{\sigma\in S_n}\sign(\sigma)\prod_{j=1}^n\zeta^{\la_a(j)\sigma(j)}
\\&=\sum_{\sigma\in S_n}\sign(\sigma)\prod_{k=1}^n\zeta^{k\sigma(\la_a^{-1}(k))}
=\sigma(\la_a)\sum_{\sigma\in S_n}\sign(\sigma\la_a^{-1})\prod_{k=1}^n\zeta^{k\sigma\la_a^{-1}(k)}.
\end{align*}
Combining this with Lemma \ref{Lem2.2}, we obtain $\det(U)=(\f an)\det(W)$, where
$W=[\zeta^{jk}]_{1\ls j,k\ls n}$.
Similarly, $\det(V)=(\f{a+1}n)\det(W^*)$, where $ W^*=[\zeta^{-jk}]_{1\ls j,k\ls n}$.
Thus, by Lemma \ref{Lem2.3} we have
$$\det(U)\det(V)=\l(\f{a(a+1)}n\r)\det(WW^*)=\l(\f{a(a+1)}n\r)\det(nI_n),$$
where $I_n=\diag(1,\ldots,1)$ is the identity matrix of order $n$.
combining this with \eqref{UCV}, we get
$$\det(Q)=c_1\cdots c_n\l(\f{a(a+1)}n\r)n^n=\l(\f{a(a+1)}n\r)(1-q^{-1})^{n-1}.$$

In view of Lemma \ref{Lem2.3},
\begin{align*}Q^{-1}&=(UCV)^{-1}=V^{-1}C^{-1}U^{-1}
\\&=\f1n [\zeta^{(a+1)jk}]_{1\ls j,k\ls n}\ \diag(c_1^{-1},\ldots,c_n^{-1})\ \f1n[\zeta^{-ajk}]_{1\ls j,k\ls n}.
\end{align*}
For any $j,k\in\{1,\ldots,n\}$, the $(j,k)$-entry of $Q^{-1}$ coincides
\begin{align*}&\ \sum_{m=1}^n \f{\zeta^{(a+1)jm}}n\times\f1{c_m}\times\f{\zeta^{-amk}}n
\\=&\ \f1{n^2}\sum_{m=1}^n\f{n(1-q^{-1/n}\zeta^{-m})}{1-q^{-1}}\zeta^{((a+1)j-ak)m}
\\=&\ \f1{n(1-q^{-1})}\(\sum_{m=1}^n\zeta^{((a+1)j-ak)m}-q^{-1/n}\sum_{m=1}^n\zeta^{((a+1)j-ak-1)m}\)
\\=&\ \f{f(j,k)}{1-q^{-1}}.
\end{align*}
This concludes our proof of Proposition \ref{Prop2.1}. \qed

\section{Proofs of Theorems \ref{Th1.1}-\ref{Th1.3}}
\setcounter{lemma}{0}
\setcounter{theorem}{0}
\setcounter{corollary}{0}
\setcounter{remark}{0}
\setcounter{equation}{0}

\medskip
\noindent{\it Proof of Theorem \ref{Th1.3} in the case $x=0$}. 
Set $x=0$. As the determinants in \eqref{q-floor} and \eqref{q-ceil}
are Laurent polynomials in $q$, it suffices to prove \eqref{q-floor} and \eqref{q-ceil}
for any fixed positive real number $q\not=1$.

 Let $A=[a_{jk}]_{1\ls j,k\ls n}$ with $$a_{jk}=q^{\lfloor\f{aj-(a+1)k}n\rfloor}=q^{\f anj}q^{-\{\f{aj-(a+1)k}n\}}q^{-\f{a+1}nk}.$$
Then $A=BQC$,
where $Q$ is given by \eqref{Q},
\begin{equation}\label{B}B=\diag(q^{a/n},q^{2a/n},\ldots,q^{na/n})
\end{equation}
and
\begin{equation}\label{C}C=\diag(q^{-(a+1)/n},q^{-2(a+1)/n},\ldots,q^{-n(a+1)/n}).
\end{equation}
Thus
$$\det(A)=\det(B)\det(Q)\det(C)=\prod_{m=1}^n q^{am/n-(a+1)m/n}\times\det(Q)$$
and hence
\begin{equation}\label{detA}\det(A)=q^{-(n+1)/2}\l(\f{a(a+1)}n\r)(1-q^{-1})^{n-1}
\end{equation}
with the aid of Proposition \ref{Prop2.1}. This proves \eqref{q-floor}.

Set $A'=[a_{jk}]_{1\ls j,k\ls n}$ with $$a'_{jk}=q^{\lceil\f{(a+1)j-ak}n\rceil}=q^{\f {a+1}nj}q^{\{\f{ak-(a+1)j}n\}}q^{-\f{a}nk}.$$
Then $A'=B'Q'C'$,
where
\begin{equation}\label{B'}B'=\diag(q^{(a+1)/n},q^{2(a+1)/n},\ldots,q^{n(a+1)/n}),
\end{equation}
\begin{equation}\label{C'}C'=\diag(q^{-a/n},q^{-2a/n},\ldots,q^{-na/n}),
\end{equation}
and
\begin{equation}\label{Q'}Q'=\l[q^{\{\f{ak-(a+1)j}n\}}\r]_{1\ls j,k\ls n}.
\end{equation}
Let $P=(Q')^T$. By Proposition \ref{Prop2.1},
$$\det(Q')=\det(P)=\l(\f{a(a+1)}n\r)(1-q)^{n-1}.$$
Note that
$$\det(A')=\det(B')\det(Q)\det(C')=\prod_{m=1}^nq^{\f{(a+1)m-am}n}\times\det(Q')$$
and hence
\begin{equation}\label{detA'}\det(A')=\l(\f{a(a+1)}n\r)q^{\f{n+1}2}(1-q)^{n-1}.
\end{equation}
This proves \eqref{q-ceil}.

In view of the above, we have completed the proof of Theorem \ref{Th1.3} in the case $x=0$. \qed

\begin{lemma}\label{Lem-rank}. Let $a$ be an integer, and let $n$ be a positive odd integer.
For the matrix $A$ and $A'$ defined in the proof of Theorem \ref{Th1.3} with $x=0$, we have
\begin{equation}\label{rank}\rank(A)\ls\f n3\ \ \t{and}\ \ \rank(A')\ls\f n3
\end{equation}
if $\gcd(a(a+1),n)>1$.
\end{lemma}
\Proof. Suppose that $\gcd(a(a+1),n)>1$. Then $n\gs3$. Note that either $\gcd(a,n)>1$ or $\gcd(a+1,n)>1$.
We just prove $\rank(A)\ls n/3$ since the inequality $\rank(A')\ls n/3$ can be proved similarly.

Suppose that $d=\gcd(a,n)>1$. Then, for each $j=r+\f nd s$ with $r\in\{1,\ldots,n/d\}$ and $s\in\{0,\ldots,d-1\}$, obviously
$a_{jk}=a_{rk}q^{as/d}$ for all $k=1,\ldots,n$. Thus $\rank (A)\ls n/d\ls n/3$. 

Now we consider the case $d=\gcd(a+1,n)>1$. Then, for each $k=r+\f nd s$ with $r\in\{1,\ldots,n/d\}$ and $s\in\{0,\ldots,d-1\}$, obviously
$a_{jk}=a_{jr}q^{-(a+1)s/d}$ for all $j=1,\ldots,n$. Thus $\rank (A)\ls n/d\ls n/3$. 

In view of the above, we have completed our proof of Lemma \ref{Lem-rank}. \qed

\begin{lemma}[The Matrix-Determinant Lemma]\label{Lem3.1}
Let $A$ be an $n\times n$ matrix over a field $F$, and let $\mathbf{u}$ and $\mathbf{v}$ be two $n$-dimensional column vectors with entries in $F$. Then
\begin{equation*}
\det( A+\u\v^T)=\det(A)+ \v^T \adj(A)\u.
\end{equation*}
\end{lemma}
\begin{remark} This lemma can be found in \cite{P} or \cite{Vr}. It is a quite useful tool to study determinant problems. For example, it plays an indispensable role in L.-Y. Wang and H.-L. Wang's proof (cf. \cite{WW}) of a challenging conjecture of Z.-W. Sun \cite{S19}.
\end{remark}

For convenience, we let $\bf 1$ denote the $n$-dimensional column vector
with all the components equal to $1$. Note that $J=[\bf1,\ldots,\bf 1]$ is the matrix of order $n$
with all entries equal to one.

\medskip
\noindent{\it Proof of Theorem \ref{Th1.1}}. As the determinant in \eqref{floor} times $(1-q)^n$
is a Laurent polynomial in $q$, it suffices to prove \eqref{floor} for any fixed positive real number $q\not=1$.

Define $A$ as in the proof of Theorem \ref{Th1.3}. Then
$$(1-q)^n\det\l[\l[\l\lfloor\f{aj-(a+1)k}n\r\rfloor\r]_q\r]_{1\ls j,k\ls n}=\det(J-A).$$
If $\gcd(a(a+1),n)>1$, then 
$$\rank(J-A)\ls \rank(J)+\rank(-A)\ls1+\f n3<n$$
with the aid of \eqref{rank}, thus $\det(J-A)=0$ and hence both sides of \eqref{floor} vanish.

Below we assume that $\gcd(a(a+1),n)=1$. By Lemma \ref{Lem3.1},
\begin{align*}\det(J-A)=&\ \det(-A+{\bf1}{\bf1}^T)=\det(-A)+{\bf 1}^T\adj(-A){\bf1}
\\=&\ (-1)^n\det(A)+{\bf1}^T((-1)^{n-1}\adj(A)){\bf 1}.
\end{align*}
As $n$ is odd, we have
\begin{equation}\label{qjk}\det\l[\l[\l\lfloor\f{aj-(a+1)k}n\r\rfloor\r]_q\r]_{1\ls j,k\ls n}
=\f{\det(J-A)}{(1-q)^n}=\f{-\det(A)+{\bf1}^T\adj(A){\bf1}}{(1-q)^n}.
\end{equation}

Note that $A$ is invertible by \eqref{detA}.
Since $\adj(A)=\det(A)A^{-1}$, we see that
${\bf 1}^T\adj(A){\bf 1}=\det(A)s$, where
$s$ is the sum of all the entries of $A^{-1}$.
Let $B,C,Q$ be given by \eqref{B}, \eqref{C} and \eqref{Q} respectively.
Then $A^{-1}=C^{-1}Q^{-1}B^{-1}$.
Combining this with \eqref{qjk}, we obtain
\begin{equation}\label{s}\det\l[\l[\l\lfloor\f{aj-(a+1)k}n\r\rfloor\r]_q\r]_{1\ls j,k\ls n}
=\f{\det(A)(s-1)}{(1-q)^n}.
\end{equation}

In light of Proposition \ref{Prop2.1},
\begin{align*}(1-q^{-1})s&=\sum_{j=1}^n\sum_{k=1}^n q^{j(a+1)/n}f(j,k)q^{-ak/n}
\\&=\sum_{j,k=1}^nq^{\f{(a+1)j-ak}n}([\![n\mid(a+1)j-ak]\!]
-q^{-1/n}[\![n\mid(a+1)j-ak-1]\!])
\\&=\sum_{1\ls j,k\ls n\atop n\mid(a+1)j-ak}q^{\f{(a+1)j-ak}n}
-\sum_{1\ls j,k\ls n\atop n\mid(a+1)j-ak-1}q^{\f{(a+1)j-ak-1}n}
\\&=\sum_{1\ls j,k\ls n\atop n\mid(a+1)j-ak}q^{\f{(a+1)j-ak}n}
-\sum_{1\ls j,k\ls n\atop n\mid(a+1)(j-1)-a(k-1)}q^{\f{(a+1)(j-1)-a(k-1)}n}
\\&=\sum_{1\ls j,k\ls n\atop n\mid(a+1)j-ak}q^{\f{(a+1)j-ak}n}
-\sum_{0\ls j,k\ls n-1\atop n\mid(a+1)j-ak}q^{\f{(a+1)j-ak}n}
\\&=q^{\f{(a+1)n-an}n}-q^{\f{(a+1)0-a0}n}=q-1
\end{align*}
and hence $s=q$. Combining this with \eqref{detA} and \eqref{s}, we immediately obtain
the desired identity \eqref{floor}.

By the above, the proof of Theorem \ref{Th1.1} is now complete. \qed

\medskip
\noindent{\it Proof of Theorem \ref{Th1.2}}. As the determinant in \eqref{ceil} times $(1-q)^n$
is a Laurent polynomial in $q$, it suffices to prove \eqref{ceil} for any fixed positive real number $q\not=1$.

Define $A'$ as in the proof of Theorem \ref{Th1.3}. If $\gcd(a(a+1),n)>1$, then 
$$\rank(J-A')\ls \rank(J)+\rank(-A')\ls 1+\f n3<n$$
by Lemma \ref{Lem-rank}, hence $\det(J-A')=0$ and both sides of \eqref{ceil} vanish.

Below we assume that $\gcd(a(a+1),n)=1$. 
Similar to \eqref{qjk}, we have
\begin{equation}\label{q'jk}\det\l[\l[\l\lceil\f{(a+1)j-ak}n\r\rceil\r]_q\r]_{1\ls j,k\ls n}
=\f{-\det(A')+{\bf1}^T\adj(A'){\bf1}}{(1-q)^n}.
\end{equation}
 
Let $P=(Q')^T$, where $Q'$ is given by \eqref{Q'}. By Proposition \ref{Prop2.1},
for $j,k=1,\ldots,n$, the $(j,k)$-entry of $P^{-1}$
is
$$\f{[\![n\mid(a+1)j-ak]\!]-q^{1/n}[\![n\mid(a+1)j-ak-1]\!]}{1-q}.$$
As $(Q')^{-1}=(P^{-1})^T$, the $(j,k)$-entry (with $1\ls j,k\ls n$) of $(Q')^{-1}$
is
$$g(j,k)=\f{[\![n\mid(a+1)k-aj]\!]-q^{1/n}[\![n\mid(a+1)k-aj-1]\!]}{1-q}.$$
Similar to \eqref{s}, we have
\begin{equation}\label{s'}\det\l[\l[\l\lceil\f{(a+1)j-ak}n\r\rceil\r]_q\r]_{1\ls j,k\ls n}
=\f{\det(A')(s'-1)}{(1-q)^n},
\end{equation}
where $s'$ is the sum of all the entries of $(A')^{-1}$.
Note that $(A')^{-1}=(C')^{-1}(Q')^{-1}(B')^{-1}$, where $B'$ and $C'$ are given by
\eqref{B'} and \eqref{C'}, respectively.
Observe that
\begin{align*}s'&=\sum_{j,k=1}^n q^{aj/n}g(j,k)q^{-(a+1)k/n}
\\&=\sum_{j,k=1}^n q^{\f{aj-(a+1)k}n}\f{[\![n\mid(a+1)k-aj]\!]-q^{1/n}[\![n\mid(a+1)k-aj-1]\!]}{1-q}
\end{align*}
and hence
\begin{align*}(1-q)s'&=\sum_{1\ls j,k\ls n\atop n\mid aj-(a+1)k}q^{\f{aj-(a+1)k}n}
-\sum_{1\ls j,k\ls n\atop n\mid aj-(a+1)k+1}q^{\f{aj-(a+1)k+1}n}
\\&=\sum_{1\ls j,k\ls n\atop n\mid aj-(a+1)k}q^{\f{aj-(a+1)k}n}
-\sum_{1\ls j,k\ls n\atop n\mid a(j-1)-(a+1)(k-1)}q^{\f{a(j-1)-(a+1)(k-1)}n}
\\&=q^{\f{an-(a+1)n}n}-q^{\f{a0-(a+1)0}n}=q^{-1}-1=\f{1-q}q.
\end{align*}
Thus $s'=q^{-1}$. Combining this with \eqref{detA'} and \eqref{s'}, we immediately obtain the desired
\eqref{ceil}. This concludes our proof of Theorem \ref{Th1.2}. \qed

\medskip
\noindent{\it Proof of Theorem \ref{Th1.3}}. In view of \cite[Lemma 2.1]{S24},  
the determinants in \eqref{q-floor} and \eqref{q-ceil} are linear in $x$. 

As we have proved \eqref{q-floor} with $x=0$, and \eqref{q-floor} with $x=-1$
follows from \eqref{floor}, we have the equality \eqref{q-floor} with $x$ general. 

Similarly, as we have proved \eqref{q-ceil} with $x=0$, and \eqref{q-ceil} with $x=-1$
follows from \eqref{ceil}, we have the equality \eqref{q-ceil} with $x$ general. 

In view of the above, our proof of Theorem \ref{Th1.3} is now complete.

\Ack. The author would like to thank his PhD student Bo Jiang for helpful comments.

\end{document}